\documentclass[a4paper,10pt,leqno]{amsart}

\usepackage{graphicx}
\usepackage{amsmath}
\usepackage{amsfonts}
\usepackage{amssymb}
\usepackage{mathtools}
\usepackage[all,cmtip]{xy}
\usepackage{amsthm}
\usepackage{comment}
\usepackage{enumitem}
\usepackage{epsfig}
\usepackage[utf8]{inputenc}
\usepackage[capitalise]{cleveref}

\makeatletter
\providecommand\@dotsep{5}
\def\listtodoname{List of Todos}
\def\listoftodos{\@starttoc{tdo}\listtodoname}
\makeatother

\newtheorem{theorem}{Theorem}[section]
\newtheorem{proposition}[theorem]{Proposition}
\newtheorem{corollary}[theorem]{Corollary}

\newtheorem{step}{Step}

  \theoremstyle{definition}
\newtheorem{definition}[theorem]{Definition}

\newtheorem{remark}[theorem]{Remark}

\newcommand{\dbR}{\mathbb{R}}

\newcommand{\dbZ}{\mathbb{Z}}

\newcommand{\calF}{{\mathcal F}}
\newcommand{\calFn}{{\mathcal F}^n}
\newcommand{\calG}{{\mathcal G}}
\newcommand{\calH}{{\mathcal H}}

\newcommand{\calO}{{\mathcal O}}
\newcommand{\calP}{{\mathcal P}}

\newcommand{\orf}[1]{O_\calF #1}

\newcommand{\vcyc}{V\text{\tiny{\textit{CYC}}}}

\newcommand{\fin}{F\text{\tiny{\textit{IN}}}}

\newcommand{\nbeq}{\begin{equation}}
\newcommand{\neeq}{\end{equation}}
\newcommand{\beq}{\begin{equation*}}
\newcommand{\eeq}{\end{equation*}}

\DeclareMathOperator{\vertex}{vert}
\DeclareMathOperator{\edge}{edge}

\DeclareMathOperator{\vrt}{vert}

\DeclareMathOperator{\cd}{cd}
\DeclareMathOperator{\vcd}{vcd}
\DeclareMathOperator{\cdfin}{\underline{cd}}
\DeclareMathOperator{\cdvc}{\underline{\underline{cd}}}

\DeclareMathOperator{\cdfn}{\cd_{\calF_n}}

\DeclareMathOperator{\gd}{gd}
\DeclareMathOperator{\gdfin}{\underline{gd}}
\DeclareMathOperator{\gdvc}{\underline{\underline{gd}}}

\DeclareMathOperator{\gdfn}{\gd_{\calF_n}}

\DeclareMathOperator{\stab}{Stab}
\DeclareMathOperator{\fix}{Fix}

\begin{document}

\title[Eilenberg-Ganea problem for families]{Groups acting on trees and the Eilenberg-Ganea problem for families}

\author{Luis Jorge S\'anchez Salda\~na}
\address{Departamento de Matemáticas,
Facultad de Ciencias, Universidad Nacional Autónoma de México, Circuito Exterior S/N, Cd. Universitaria, Colonia Copilco el Bajo, Delegación Coyoacán, 04510, México D.F., Mexico}
\email{luisjorge@ciencias.unam.mx}


\date{}


\begin{abstract}
We construct new examples of groups with cohomological dimension 2 and geometric dimension 3 with respect to the families of finite subgroups, virtually abelian groups of bounded rank, and the family of virtually poly-cyclic subgroups. Our main ingredients are the examples constructed by Brady-Leary-Nucinckis and Fluch-Leary,  and Bass-Serre theory. 
\end{abstract}
\maketitle

\section{Introduction}

Let $G$ be a discrete group.  A non-empty collection $\calF$ of subgroups of
$G$ is called a \emph{family} if it
is closed under conjugation and under taking subgroups. We
call a $G$-CW-complex $X$ a \emph{model} for \emph{the classifying space} $E_{\calF}G$ if the following conditions are satisfied:
\begin{enumerate}
  \item For all $x\in X$, the isotropy group $G_x$ belongs to $\calF$.

  \item For all  $H\in\calF$ the subcomplex $X^H$  of $X$, consisting of points in $X$ that are fixed under all elements of $H$, is contractible. In particular $X^H$ is non-empty.
\end{enumerate}
Equivalently, $X$ is a model for $E_\calF G$ if, for every $G$-CW-complex $Y$ with isotropy groups in $\calF$, there exists a $G$-map, unique up to $G$-homotopy, $Y\to X$.

Such a model always exists for every discrete group $G$ and every family of subgroups $\calF$ of $G$. Moreover, every two models for $E_\calF G$ are $G$-homotopically equivalent. The \emph{$\calF$-geometric dimension of $G$}, denoted $\gd_\calF(G)$, is the minimum $n$ for which there exists an $n$-dimensional model for $E_\calF G$.

On the other hand, given $G$ and $\calF$ we have the so-called restricted orbit category $\calO_\calF G$, which has as objects the homogeneous $G$-spaces $G/H$, $H\in \calF$, and morphisms are  $G$-maps. A  \emph{$\calO_\calF G$-module} is a contraviariant functor from $\calO_\calF G$ to the category of abelian groups, and a morphism between two $\calO_\calF G$-modules is a natural transformation of the underlying functors. Denote by $\calO_\calF G$-mod the category of $\calO_\calF G$-modules. It turns out that $\calO_\calF G$-mod is an abelian category with enough projectives. Thus we can define a $G$-cohomology theory for $G$-spaces $H_\calF^*(-;M)$ for every $\calO_\calF G$-module $M$ (see \cite[pg.~7]{MV03}).  The
\emph{$\calF$-cohomological dimension of $G$}---denoted $\cd_{\calF}(G)$---is the
largest non-negative $n\in\dbZ$ for which the cohomology group $H^n_{\calF}(G;M)=H^n_{\calF}(E_\calF G;M)$
is nontrivial for some $M\in\mathcal{O}_{\calF}G \textrm{-mod}$. Equivalently, $\cd_\calF(G)$ is the length of the shortest projective resolution of the constant $\orf{G}$-module $\dbZ_\calF$, where $\dbZ_\calF$ is given by $\dbZ_\calF(G/H)=\dbZ$, for all $H\in \calF$, and every morphism of $\orf{G}$ goes to the identity function.

From \cite{LM00} we have the following Eilenberg-Ganea type theorem. For every group $G$ and every family of subgroups $\calF$ we have
    \[
    \cd_\calF(G)\leq \gd_\calF(G)\leq \max\{3,\cd_\calF(G)\}.
    \]
Therefore, if $\cd_\calF(G)\geq 3$, then $\cd_\calF(G)=\gd_\calF(G)$.  The first inequality actually holds for any group $G$, and any family $\calF$, since the augmented Bredon cellular chain complex of any model for  $E_\calF G$ is a free resolution of $\dbZ_\calF$.

It is easy to prove that $\cd_\calF(G)=0$ if and only if $\gd_\calF(G)=0$. In dimension $1$ things seem to be harder (see \cite{Deg17} and \cite{LO18}). It is worth saying that for the trivial family, the family of finite subgroups and the family of virtually cyclic subgroups (provided $G$ is countable) it is known that  $\cd_\calF(G)=1$ if and only if $\gd_\calF(G)=1$. Still, in full generality, it is open the possibility of  having $1\leq\cd_\calF(G)<\gd_\calF(G)\leq 3$, for some group $G$ and some family $\calF$.

For the families $\fin$ and $\vcyc$ of finite and virtually cyclic subgroups respectively, there are examples due to Brady-Leary-Nucinkis \cite{BLN01}, and Fluch-Leary \cite{FL14} of groups $G$ with cohomological dimension 2 and geometric dimension 3.
Explicitly, let $G$ be a  right angled Coxeter group such that the nerve of the group 
is $2$-dimensional, acyclic, and cannot be embedded in any contractible $2$-complex, then $\cdfin(G)=2$ and $\gdfin(G)=3$. If, additionally, $G$ is word-hyperbolic, then $\cdvc(G)=2$ and $\gdvc(G)=3$. Also in  \cite{LP17}, Leary-Petrosyan generalized and strengthened the examples given by Brady-Leary-Nucinkis.

In the present note we construct new  examples, out of known examples, of groups satisfying $\cd_\calF(G)=2$ and $\gd_\calF(G)=3$ for the family of finite subgroups, the family of virtually abelian groups of bounded rank, and the family of virtually poly-cyclic subgroups of bounded rank. Our examples are fundamental groups of graphs of groups (in the sense of Bass and Serre), subject to certain restrictions on the vertex and edge groups, and the action of the group on its Bass-Serre tree. 

\vskip 10pt

The present paper is organized as follows. In Section \ref{sec:prelim} we recall the notation of graphs of groups and we give a procedure to construct classifying spaces for fundamental groups of graphs of groups. In Section \ref{sec:main:theorem} we state and prove our main result, \Cref{main:thm}. Finally in Section \ref{section:examples} we give some applications of our main theorem.

\subsection*{Acknowledgements}
 The author wishes to  thank Jean-Fran\c cois Lafont for many useful discussions and feedback on early versions of this paper, Eduardo Mar\-t\'inez-Pedroza for some useful conversations, and Nansen Petrosyan for pointing out reference \cite{LP17}. The author thanks to Arilín Haro for uncountable stimulating conversations. Also the author thanks the anonymous referee for useful comments that helped to improve the exposition of the present paper. This work was funded by the Mexican Council of Science and Technology via the program \textit{Estancias postdoctorales en el extranjero}, and  by the NSF, via grant DMS-1812028.

 \section{Preliminaries}\label{sec:prelim}

\subsection{Graphs of groups} In this subsection we give a quick review of Bass-Serre theory, referring the reader to \cite{Se03} for more details. A graph (in the sense of Bass and Serre) consists of a set of vertices $V=\vertex Y$, a set of (oriented) edges $E=\edge Y$, and two maps $E\to V\times V$, $y\mapsto (o(y),t(y))$, and $E\to E$, $y\mapsto \overline{y}$ satisfying $\overline{\overline{y}}=y$, $\overline{y}\neq y$, and $o(y)=t(\overline{y})$. The vertex $o(y)$ is called the \emph{origin} of $y$, and the vertex $t(y)$ is called the \emph{terminus} of $y$.

An orientation of a graph $Y$ is a subset $E_+$ of $E$ such that $E=E_+\bigsqcup \overline{E}_+$. We can define  \emph{path} and \emph{circuit} in the obvious way.

A \emph{graph of groups} $\mathbf{Y}$ consists of a graph $Y$, a group $Y_P$ for each $P\in \vertex Y$, and a group $Y_y$ for each $y\in \edge Y$, together with monomorphisms $Y_y\to Y_{t(y)}$. One requires in addition $Y_{\overline{y}}=Y_y$.

Suppose that the group $G$ acts without inversions on a graph $X$, i.e. for every $g\in G$ and $x\in \edge X$ we have $g x\neq\overline{x} $. Then we have an induced graph of groups with underlying graph $X/G$ by  associating to each vertex (resp. edge) the isotropy group of a preimage under the quotient map $X\to X/G$ (see \cite[\S~1.5.4]{Se03}).

Given a graph of groups $\mathbf{Y}$, one of the classic theorems of Bass-Serre theory provides the existence of a group $G=\pi_1(\mathbf{Y})$, called the \emph{fundamental group of the graph of groups $\mathbf{Y}$} and a tree $T$ (a graph with no cycles), called the \emph{Bass-Serre tree of $\mathbf{Y}$}, such that $G$ acts on $T$ without inversions, and the induced graph of groups is isomorphic to $\mathbf{Y}$ (see \cite[\S~1.5.3~Theorem~12]{Se03}). The identification $G=\pi_1(\mathbf{Y})$ is called a \emph{splitting} of $G$.

\subsection{Graph of groups and classifying spaces}
Analogously to a graph of groups we can define a graph of spaces $\mathbf{X}$ as a graph $X$, connected CW-complexes $X_P$ and $X_y$ for each vertex $P$ and each edge $y$, and closed cellular embeddings $X_y\to X_{t(y)}$. 
Then we will have a CW-complex, called the \emph{geometric realization}, that is assembled by gluing the ends of the product space $X_y\times [0,1]$ to the spaces $X_{t(y)}$, $X_{o(y)}$.

Given a graph of spaces $\mathbf{X}$ with $\pi_1(X_y)\to \pi_1(X_P)$ injective, there is an associated graph of groups $\mathbf{Y}$ with the same underlying graph and whose vertex (resp. edge) groups are the fundamental groups of the corresponding vertex (resp. edge) CW-complexes. Then, as a generalization of the Seifert--van Kampen theorem, we have that the fundamental group of the geometric realization of $\mathbf{X} $ is naturally isomorphic to the fundamental group of the graph of groups $\mathbf{Y}$.

\vskip 10pt

Let $\mathbf{Y}$ be a graph of groups with fundamental group $G$, and let $\calF$ be a family of subgroups of $G$. Let $T$ be the Bass-Serre tree of $\mathbf{Y}$. We are going to construct a graph of spaces $\mathbf{X}$, with underlying graph $T$, using the classifying spaces of the edge groups and the vertex groups, and the corresponding families $Y_y\cap\calF$ and $Y_P\cap \calF$. Let $X_{P}$ be a model for the classifying space $E_{Y_P\cap\calF} Y_P$, and $X_{y}$ be a model for $E_{Y_y\cap\calF} Y_y$, for every vertex $P$ and edge $y$ of $Y$. Hence for every monomorphism $Y_y\to Y_{t(y)}$ we have a $Y_y$-equivariant cellular map (unique up to $Y_y$-homotopy) $X_y \to X_{t(y)}$ which leads to the $G$-equivariant cellular map $G\times_{Y_y} X_y \to G\times_{Y_{t(y)}} X_{t(y)} $. This gives us the information required to define a graph of spaces $\mathbf{X}$ with underlying graph $T$, the Bass-Serre tree of $\mathbf{Y}$. Moreover, we have a cellular $G$-action (actually a $G$-CW-structure) on the geometric realization $X$ of $\mathbf{X}$.

Let us describe $X$ in greater detail. Consider the geometric realization of the Bass-Serre tree of $\mathbf{Y}$, and abusing of notation denote it by $T$. Hence $T$ is a $1$-dimensional and contractible CW-complex. Also $G$ acts on $T$ by cellular automorphism and the quotient space $T/G$ is isomorphic to the geometric realization of $Y$ (we denote this geometric realization by $Y$). Hence the $0$-skeleton of $T$ can be identified with $\coprod_{P\in \vertex Y} G/G_P=\coprod_{P\in \vertex Y} G\times_{G_P}\{*\}$, where $\{*\}$ is the one-point space, and $T$ is given by  the following $G$-pushout

\[\xymatrix{\coprod_{y\in E_+} G\times_{G_y}\{0,1\} \ar[d] \ar[r]  & \coprod_{P\in \vertex Y} G\times_{G_P}\{*\}\ar[d]\\
\coprod_{y\in E_+} G\times_{G_y}[0,1]\ar[r] & T
}
\]
where the left vertical map is the inclusion map, and the upper horizontal map is the disjoint union of the attaching maps of the one cells, and $E_+$ is an orientation for the graph $Y$ (considered as a graph in the sense of Serre).

Next we modify the above pushout to obtain $X$. The modification consists on replacing every $0$-cell $x$ of $T$ by the chosen model for the classifying space $E_{\calF\cap G_x}G_x$, and every $1$-cell $I$ with isotropy group $H$ by $I\times E_{\calF\cap H}H$. Explicitly, $X$ is given by the following $G$-pushout

\begin{equation}\label{pushout:bass:serre}
    \xymatrix{\coprod_{y\in E_+} G\times_{G_y}(\{0,1\}\times X_y) \ar[d] \ar[r]  & \coprod_{P\in \vertex Y} G\times_{G_P}\times X_P\ar[d]\\
\coprod_{y\in E_+} G\times_{G_y}([0,1]\times X_y)\ar[r] & X
}
\end{equation}

The proof of the following result is completely analogous to \cite[Proposition~4.7]{JLSS}.

\begin{proposition}
\label{bass serre construction}
The geometric realization $X$ of the graph of spaces $\mathbf{X}$ constructed above is a model for $E_{\calF'}G$, where $\calF'$ is the family of all elements of $\calF$ that are conjugated to a subgroup in  $Y_P$, for some $P\in \vertex Y$. In particular, there exists a model for $E_{\calF'} G$ of dimension  \[ \max\{\gd_{\calF\cap Y_y}(Y_y)+1,\gd_{\calF\cap Y_P}(Y_P)| y\in \edge Y, P\in \vertex Y \}. \]
\end{proposition}

\begin{remark}\label{remark:F0}
Note that $\calF'$ can also be described as the subset of $\calF$ whose elements fix one vertex of $T$.
\end{remark}

\section{Main theorem}\label{sec:main:theorem}

In order to state our main theorem we need to introduce some notation.

\vskip 10pt

Let $G$ be a group and $\calF$ be a family of subgroups of $G$. We say that $G$ is an \emph{$\calF$-Eilenberg-Ganea group} if $\gd_\calF(G)=3$ and $\cd_\calF(G)=2$.

\vskip 10pt

Let $\mathbf{Y}$ be a graph of groups  and let $\calF$ be a family of subgroups of $G=\pi_1(\mathbf{Y})$. We say $\mathbf{Y}$ is \emph{$\calF$-admissible} if the following conditions hold:
\begin{itemize}
    \item There exists a vertex $P$ such that $Y_P$ is a $(Y_P\cap \calF)$-Eilenberg-Ganea group.
    \item  For all vertices $P$ of $\mathbf{Y}$ we have $\gd_{\calF\cap Y_P}(Y_P)\leq 3$ and $\cd_{\calF\cap Y_P}(Y_P)\leq 2$.
    \item  For all edges $y$ of $\mathbf{Y}$ we have $\gd_{\calF\cap Y_y}(Y_y)\leq 2$ and $\cd_{\calF\cap Y_y}(Y_y)\leq 1$.
\end{itemize}

\begin{definition}\label{rel:acylindrical:def}
Let  $\mathbf{Y}$ be a graph of groups with fundamental group $G$, let $\calP$  collection of subgroups of $G$, and let $T$ be the Bass-Serre tree of $\mathbf{Y}$. The splitting of $G$ is said to be \emph{$\calP$-acylindrical} if there exists an integer $k$ such that, for every path $c$ of length $k$ in $T$, the stabilizer of $c$ belongs to $\calP$.
\end{definition}

Let $G$ and $Q$ be groups, and let $\calP$ be a collection of subgroups of $G$. A $\calP$-by-$Q$ group is a group $L$ that fits in a short exact sequence
\[1\to N \to L \to Q \to 1\]
for some $N$ in $\calP$.

The following is a result that will be helpful to give examples of $\fin$-Eilenberg-Ganea groups in \Cref{section:examples}.

\begin{theorem}\label{premain:thm}
Let $\mathbf{Y}$ be a graph of groups  with fundamental group $G$ and Bass-Serre tree $T$. Let $\calF$ be a family of subgroups of $G$. Assume
\begin{enumerate}[label=(\alph*)]
    \item\label{premain:thm:admissible} $\mathbf{Y}$ is $\calF$-admissible; and 
    \item\label{premain:thm:f0} every element of $\calF$ fixes a vertex in $T$.
\end{enumerate}
Then $G$ is a $\calF$-Eilenberg-Ganea group.
\end{theorem}
\begin{proof}
Choose models $X_P$ (resp. $X_y$) for $E_{Y_P\cap\calF}Y_P$ (resp. $E_{Y_y\cap\calF}Y_y$) of minimal dimension for every vertex $P$ (resp. edge $y$) of $\mathbf{Y}$. Then by \Cref{bass serre construction} and hypothesis \ref{premain:thm:admissible} and \ref{premain:thm:f0}
we obtain a $3$-dimensional model $X$ for $E_{\calF} G$. Therefore $\cd_\calF(G)\leq\gd_\calF(G)\leq 3$.

We claim that $T$ is a model for $E_\calG G$, where $\calG$ is the family of subgroups of $G$ generated by the isotropy groups of $T$. In fact, since $T$ is a tree, the fixed point subspace $T^H$ is either empty or a tree, in particular contractible, for every subgroup $H$ of $G$ (see \cite[pg.~58]{Se03}). Moreover $T^H$ is non-empty if and only if $H$ is contained in the isotropy group of a point of $T$. Now the claim follows.

Hypothesis \ref{premain:thm:f0} implies $\calF\subseteq \calG$. Hence \cite[Remark 4.2]{MP02} yields the
long exact sequence
\begin{equation*}
\begin{multlined}
\cdots{\to} H_{\calF}^2(G;M) {\to} \bigoplus_{\vrt Y}H^2_{\calF\cap Y_P}(Y_P;M)
{\to} \bigoplus_{{y}\in E_+}H_{\calF\cap Y_y}^2(Y_y;M)
{\to}\\ H^3_{\calF}(G;M) {\to}
\bigoplus_{\vrt Y}H_{\calF\cap Y_P}^3(Y_P;M)\to\cdots
\end{multlined}
\end{equation*}
where $E_+$ is an orientation of $Y$. Note that the sums run over $G$-orbits of $0$- and $1$-cells of $T$ (compare with the pushout (\ref{pushout:bass:serre})).

By hypothesis \eqref{main:thm:admissible} $\cd_{\calF\cap Y_P}(Y_P)\leq 2$ and $\cd_{\calF\cap Y_y}(Y_y)\leq 1$ for every vertex $P$ and every edge $y$ of $\mathbf{Y}$. Thus the  third term and last term in the long exact sequence vanish.
Hence  $H^3_{\calF}(G;M)=0$, for every $\orf{G}$-module $M$. Therefore $\cd_{\calF}(G)\leq 2$.

By hypothesis \ref{premain:thm:admissible}, there exists a vertex $P$ of $\mathbf{Y}$ such that $Y_P$ is a $(Y_P\cap \calF)$-Eilenberg-Ganea group. Hence we have
\[ 3=\gd_{\calF\cap Y_P}(Y_P)\leq \gd_\calF (G) \]
and
\[ 2=\cd_{\calF\cap Y_P}(Y_P)\leq \cd_\calF (G). \]

\end{proof}

Now we state and prove our main result. This theorem will be used to give more examples of $\calF$-Eilenberg-Ganea groups in \cref{section:examples}.

\begin{theorem}\label{main:thm}
Let $\mathbf{Y}$ be a graph of groups  with fundamental group $G$ and Bass-Serre tree $T$. Let $\calF$ be a family of subgroups of $G$. Denote by $\calF_0$ (resp. $\calF_1$) the collection of all elements of $\calF$ that fix a vertex (resp. act cocompactly on a geodesic line) of $T$. Assume 
\begin{enumerate}
    \item\label{main:thm:admissible} $\mathbf{Y}$ is $\calF$-admissible.
    \item\label{main:thm:f0cupf1} $\calF=\calF_0\sqcup\calF_1$;
    \item\label{main:thm:acylindrical} the splitting of $G$ is $\calP$-acylindrical for some subgroup closed collection  $\calP\subseteq \calF_0$ of subgroups of $G$;
    \item\label{main:thm:pbyz} every $\calP$-by-$(\dbZ$ or $D_\infty)$ subgroup of $G$ belongs to $\calF$; and
    \item\label{main:thm:f0byZ2} every $\calF_0$-by-$\dbZ/2$ subgroup of $G$ belongs to $\calF_0$.
\end{enumerate}
Then $G$ is an $\calF$-Eilenberg-Ganea group.
\end{theorem}
\begin{proof}
We will prove this theorem in several steps.
\begin{step}\label{step:gd:less:3}
There exists a $3$-dimensional model for $E_{\calF_0 }G$. 
\end{step}

Choose models $X_P$ (resp. $X_y$) for $E_{Y_P\cap\calF}Y_P$ (resp. $E_{Y_y\cap\calF}Y_y$) of minimal dimension for every vertex $P$ (resp. edge $y$) of $\mathbf{Y}$. Then by \Cref{bass serre construction} (see also \Cref{remark:F0}) and hypothesis \eqref{main:thm:admissible} we obtain a $3$-dimensional model $X$ for $E_{\calF_0} G$.

\begin{step}\label{step:attaching:2cells}
We can attach cells of dimension at most $2$ to $X$ to obtain a $3$-dimensional model $Y$ for $E_{\calF} G$. In particular $\gd_{\calF}(G)\leq 3$. 
\end{step}

Let $H$ be an element of $\calF_1$. Denote by $\gamma_H$ the geodesic line of $T$ on which $H$ acts cocompactly. Also denote by $\stab_G(\gamma_H)$ and $\fix_G(\gamma_H)$ the setwise and the pointwise stabilizers of $\gamma_H$ in $G$ respectively. Then we have a short exact sequence
\[1\to \fix_G(\gamma_H) \to \stab_G(\gamma_H) \to D \to 1\]
where $D$ is isomorphic to $\dbZ$ or $D_\infty$. Since $\gamma_H$ contains paths of arbitrary large length,  by hypothesis \eqref{main:thm:acylindrical} $\fix_G(\gamma_H)$ belongs to $\calP\subseteq\calF_0$. By hypothesis \eqref{main:thm:pbyz} $\stab_G(\gamma_H)$ belongs to $\calF_1\subseteq \calF$.

Every isotropy group of the action of $\stab_G(\gamma_H)$ on $\gamma_H\cong \dbR$ is isomorphic to a $\fix_G(\gamma_H)$-by-(1 or $\dbZ/2$) subgroup of $\stab_G(\gamma_H)$. Such an action can be seen as the action induced by the surjection $\stab_G(\gamma_H) \to D$. Thus by hypothesis \eqref{main:thm:f0byZ2} we have that $\gamma_H\cong \dbR$ is a model for $E_{\calF_0\cap \stab_G(\gamma_H)} \stab_G(\gamma_H)$.

Define $\calH$ to be the set of all geodesics $\gamma_H$ of $T$ for $H\in \calF_1$ (note that $\calH$ is in bijection with  the collection of all maximal elements of $\calF_1$). For every $\gamma_H$ in $\calH$ we have an $H$-map $E_{\calF_0\cap H} H=\gamma_H\to X$, which induces a $G$-map $G\times_H\gamma_H\to X$. The latter, from now on, will be assumed to be an inclusion by replacing $X$ with the mapping cylinder of $G\times_H\gamma_H\to X$ if necessary.  

Denote by \{*\} the one-point-space.  Let $Y$ be the $G$-CW-complex given by the following homotopy $G$-pushout

\[\xymatrix{\coprod_{\gamma_H\in\calH} G\times_{H}\gamma_H \ar[d] \ar[r]  & X\ar[d]\\
\coprod_{\gamma_H\in\calH}G\times_{H}\{*\}\ar[r] & Y
}
\]
Hence, $Y$ is essentially obtained from $X$ coning-off all geodesics $\gamma_H$, for $H\in \calH$. In consequence $Y$ is obtained from $X$ by attaching cells of dimension at most $2$.

Next we will verify that $Y$ is a model for $E_\calF G$. 

Note that every point $x$ of $Y$ is either a point of $X$ or the vertex of a coned geodesic $\gamma_H$. In the latter case we will say $x$ is a conic point. If $x$ is a point of $X$, then $G_x$ is an element of $\calF_0$. If $x$ is conic point of $Y$, then $G_x$ has the form $\stab_G(\gamma_H)$, thus is an element of $\calF_1$. Therefore for every $x\in X$, the isotropy group $G_x$ belongs to $\calF$.

Let $K$ be an element in $\calF$. If $K$ belongs to $\calF_0$, then $Y^K$ can be obtained from $X^K$ coning-off some geodesic segments (those corresponding to $X^K\cap \gamma_H^K$). Thus $Y^K$ is contractible. If $K$ is an element of $\calF_1$, then $Y^K$ is a union of conic points. It is enough to observe that $K$ can only act cocompactly on one geodesic line of $T$. In fact, since $K$ acts cocompactly on a geodesic of $T$, it has to contain a hyperbolic isometry $g$ of $T$. It is well known that a hyperbolic isometry of $T$ acts by translation on a unique geodesic of $T$ (because $T$ is a tree). Thus the geodesic upon which $K$ is acting is unique. Hence $Y^K$ consists of only one conic point. Therefore for every $K$ in $\calF$, $Y^K$ is contractible. This finishes the proof that $Y$ is a model for $E_\calF G$.

\begin{step}
There is an $\calO_\calF G$-projective resolution of $\dbZ_\calF$ of length $2$. It follows that  $\cd_\calF (G)\leq 2$.
\end{step}

To prove $\cd_\calF(G)\leq 2$ we can use a similar argument to \cite[Theorem~7]{FL14}, to construct a projective resolution of $\dbZ_{\calF}$ of length $2$. We include this argument for completeness. Consider the following short exact sequence of $\calO_\calF G$-modules

\[1\to \bar\dbZ_{\calF_0}\to \dbZ_{\calF} \to Q\to 1\]
where $Q$ is the $\calO_\calF G$-module given by $Q(G/H)=\dbZ$ if $H\in \calF_1$ and $Q(G/H)=0$ if $H\in \calF_0$.

If we find projective resolutions for $\bar\dbZ_{\calF_0}$ and $Q$ of length   $2$, we can apply the Horseshoe lemma to obtain a projective resolution for $\dbZ_{\calF}$ of length $2$. 

By \Cref{step:attaching:2cells} we obtain a model $Y$ for $E_{\calF}G$ from a model $X$ for $E_{\calF_0}G$ by attaching cells of dimension at most $2$. Then the Bredon cellular chain complex $C_*(Y,X)$ gives a length $2$ projective resolution for $Q$.

Next, we will prove that $\cd_{\calF_0}(G)\leq 2$. 

 Using an argument similar to that in the proof of \Cref{premain:thm}, we get the
long exact sequence
\begin{equation*}
\begin{multlined}
\cdots{\to} H_{\calF_0}^2(G;M) {\to} \bigoplus_{\vrt Y}H^2_{\calF_0\cap Y_P}(Y_P;M)
{\to} \bigoplus_{{y}\in E_+}H_{\calF_0\cap Y_y}^2(Y_y;M)
{\to}\\ H^3_{\calF_0}(G;M) {\to}
\bigoplus_{\vrt Y}H_{\calF_0\cap Y_P}^3(Y_P;M)\to\cdots
\end{multlined}
\end{equation*}

Note that $\calF\cap Y_P=\calF_0\cap Y_P$ and $\calF\cap Y_y=\calF_0\cap Y_y$ for every vertex $P$ and every edge $y$ of $\mathbf{Y}$. By hypothesis \eqref{main:thm:admissible} $\cd_{\calF\cap Y_P}(Y_P)\leq 2$ and $\cd_{\calF\cap Y_y}(Y_y)\leq 1$ for every vertex $P$ and every edge $y$ of $\mathbf{Y}$. Thus the  third term and last term vanish.
Therefore  $H^3_{\calF_0}(G;M)=0$, for every $\calO_{\calF_0} G$-module $M$. Hence $\cd_{\calF_0}(G)\leq 2$ and, as a consequence, there exists a projective resolution of $\calO_{\calF_0} G$-modules for $\dbZ_{\calF_0}$ of length two:
\[0\to P_2\to P_1\to P_0\to \dbZ_{\calF_0}\to 0\]
We still have to promote this projective resolution of $\calO_{\calF_0} G$-modules to a resolution for $\bar\dbZ_{\calF_0}$ in the category of $\calO_{\calF} G$-modules. For this consider the natural inclusion functor $\calO_{\calF_0} G\to \calO_{\calF} G$ and the associated induction functor $\mathrm{ind}_{\calF_0}^\calF$, thus by \cite[Lemma~7.3]{DP14} $\mathrm{ind}_{\calF_0}^\calF(\dbZ_{\calF_0})=\bar \dbZ_{\calF_0}$ and by \cite[Lemma~7.5]{DP14}
\[0\to \mathrm{ind}_{\calF_0}^\calF{P_2}\to \mathrm{ind}_{\calF_0}^\calF{P_1}\to \mathrm{ind}_{\calF_0}^\calF{P_0}\to \mathrm{ind}_{\calF_0}^\calF(\dbZ_{\calF_0})=\bar\dbZ_{\calF_0}\to 0\]
is a projective resolution of $\calO_{\calF} G$-modules of length 2.

\begin{step}
We have $\cd_\calF (G)\geq 2$ and $\gd_\calF (G)\geq 3$.
\end{step}

By hypothesis \ref{main:thm:admissible}, there exists a vertex $P$ of $\mathbf{Y}$ such that $Y_P$ is a $(Y_P\cap \calF)$-Eilenberg-Ganea group. Hence we have
\[ 3=\gd_{\calF\cap Y_P}(Y_P)\leq \gd_\calF (G) \]
and
\[ 2=\cd_{\calF\cap Y_P}(Y_P)\leq \cd_\calF (G). \]

This finishes the whole proof.

\end{proof}

\begin{remark}
Our definition of $\calP$-acylindrical splitting is a natural generalization of the notion of acylindrical splitting in \cite{LO09}. This property allows us to gain control on the stabilizers of geodesics, which is a very important part of our proof.
\end{remark}

\begin{remark}\label{remark:slender} The following observations concern hypothesis (2) of \Cref{main:thm}.
\begin{enumerate}
    \item Not every family $\calF$ of $G$ will satisfy hypothesis \eqref{main:thm:f0cupf1} in \cref{main:thm}. In fact, we may have a subgroup of $G$ that neither fixes a point nor acts on a geodesic by translations.  For instance, let $F$ be a subgroup of $G$ generated by two elements $a$ and $b$. Assume that $a$ and $b$ act as non-trivial translations on two distinct geodesics $l_a$ and $l_b$ of $T$ respectively. By \cite[Proposition~25]{Se03}, neither $a$ nor $b$ have a fixed point in $T$. Therefore $F$ does not fix a point in $T$. On the other hand, as a consequence of \cite[Proposition~24(iii)]{Se03} each $a$ and $b$ stabilize a unique geodesic, therefore $F$ cannot stabilize a geodesic because $l_a\neq l_b$.

\item Note that no subgroup of $G$ fixes a point \emph{and} acts on a geodesic by non-trivial translations, i.e. $\calF_0\cap \calF_1=\emptyset$. Therefore, in hypothesis (2) of \Cref{main:thm}, we can actually replace the disjoint union by the union $\calF_0\cup \calF_1=\calF$ . To prove this claim assume $H$ is a subgroup of $G$ that fixes a point, then by \cite[Proposition~25]{Se03}, no element of $H$ acts as a translation on a geodesic of $T$. If $H$ acts by translation on a geodesic, then at least one element of $T$ should acts as a translation on that geodesic, which is a contradiction. Therefore no element in $\calF_0$ can be an element in $\calF_1$. 

\item From \cite[Lemma~1.1]{DS99} we know that, if $H$ is a \emph{slender} group (every subgroup of $H$ is finitely generated) acting on a tree, then either it fixes a vertex or it stabilizes a geodesic. Thus, if every element of $\calF$ is slender, then we have that $\calF$ satisfies hypothesis \eqref{main:thm:f0cupf1} in \cref{main:thm}. 
Examples of slender groups are: finite groups and virtually poly-cyclic groups. This is one of the reasons our theorem applies to the examples in Section \ref{section:examples}.
\end{enumerate}
\end{remark}

\section{Examples}\label{section:examples}

In this section we give some applications of \Cref{premain:thm} and \Cref{main:thm}.

\subsection{The family of finite subgroups} 
Recall that, given $G$, $\fin$ denotes the family of finite subgroups of $G$. Also we denote by $\cdfin(G)$, $\gdfin(G)$, and $\underline{E}G$ the $\fin$-cohomological dimension, $\fin$-geometric dimension, and classifying space of $G$ with respect to the family of finite subgroups, respectively.

\begin{theorem}\label{fin:eilenberg:ganea}
Let $\mathbf{Y}$ be a graph of groups  with fundamental group $G$ and Bass-Serre tree $T$. Let $\fin$ be a family of finite subgroups of $G$. Assume $\mathbf{Y}$ is $\fin$-admissible. Then $G$ is a $\fin$-Eilenberg-Ganea group.
\end{theorem}
\begin{proof}
It is well known that every finite group acting on a tree has a fixed point. Thus this theorem is a direct consequence of \Cref{premain:thm}.
\end{proof}

\begin{remark}
Notice that it is not difficult to construct an $\fin$-admissible graph of groups $\mathbf{Y}$. In fact, we can place one of the groups described in \cite{BLN01} in one of the vertex of $\mathbf{Y}$. For all the other vertex groups $Y_P$ and edge groups $Y_y$ we have to choose groups such that $\underline{\gd}(Y_P)\leq 3$ and $\underline{\cd}(Y_P)\leq 2$, and $\underline{\gd}(Y_y)\leq 2$ and $\underline{\cd}(Y_y)\leq 1$. Examples of groups satisfying these conditions are finite groups  and virtually free groups. 
\end{remark}

\begin{remark}
As a direct consequence of \Cref{fin:eilenberg:ganea} we have that the collection of $\fin$-Eilenberg-Ganea groups is closed under taking amalgamated products with finite amalgam, and it is closed under taking HNN-extensions over finite subgroups. More explicitelly, if $G$ and $H$ are $\fin$-Eilenberg-Ganea groups and $N$ is a common finite subgroup, then $G*_N H$ and $G*_N$ are $\fin$-Eilenberg-Ganea groups.
\end{remark}

\subsection{The family of virtually abelian groups of rank at most $r$}

Let $G$ be a discrete group and let $n\geq0$ be an integer. We say a group is \emph{$\dbZ^r$-group} (resp. \emph{virtually $\dbZ^r$}) if it is isomorphic to $\dbZ^r$ (resp. contains a finite index subgroup isomorphic to $\dbZ^r$). Define the following family of subgroups of $G$
\[
\calF^n = \{ K\leq G | K \text{ is virtually }\dbZ^r \text{ for some }r\leq n  \}
\]

Note that $\calF^0=\fin$ and $\calF^1=\vcyc$. The family $\calFn$ has been recently studied in \cite{LO18}, \cite{HP19}, \cite{Pr18}.

\begin{theorem}\label{vituallyZ:theorem}
Let $\mathbf{Y}$ be a graph of groups  with fundamental group $G$. Let $n\geq1$.  Assume the following conditions hold:
\begin{enumerate}[label=(\Alph*)]
    \item\label{virtuallyZ:theorem:admissible} $\mathbf{Y}$ is $\calFn$-admissible.
    \item\label{virtuallyZ:theorem:acylindrical} The splitting of $G$ given by $\mathbf{Y}$ is $\fin$-acylindrical.
\end{enumerate}
Then $G$ is an $\calF^n$-Eilenberg-Ganea group.
\end{theorem}
\begin{proof}
We will see that this theorem is a consequence of \Cref{main:thm}. By \Cref{remark:slender}(3), every virtually $\dbZ^r$ group acting on a tree either fixes a vertex or it acts cocompactly on a geodesic. Thus, in notation of \Cref{main:thm} we have $\calFn=\calFn_0\sqcup \calFn_1$. Moreover every element of $\calFn_1$ is a virtually cyclic subgroup by hypothesis \ref{virtuallyZ:theorem:acylindrical}.

In the statement of \Cref{main:thm} we choose $\calP=\fin$. Every $\fin$-by-$(\dbZ$ or $D_\infty)$ is virtually cyclic, hence belongs to $\calFn$. Clearly any $\fin$-by-$\dbZ/2$ is finite. Thus we have verified all hypothesis of \Cref{main:thm}, and the result follows. 
\end{proof}

\begin{corollary}\label{corollary:finite:edge:stab}
Let $\mathbf{Y}$ be a graph of groups and let $G=\pi_1(\mathbf{Y})$. Let $n\geq 0$ be an integer.  Assume that $\mathbf{Y}$ is $\calFn$-admissible with  finite edge groups.
Then $G$ is an $\calFn$-Eilenberg-Ganea group.
\end{corollary}
\begin{proof}
Since $Y_y$ is finite for all edge $y$, it is easy to see that $\gdfn(Y_y)=\cdfn(Y_y)=0$. Since the edge groups are finite, the splitting of $G$ is $\fin$-acylindrical. Thus the result follows from \Cref{vituallyZ:theorem}.

\end{proof}

\begin{remark}
Note that, the $\vcyc$-Eilenberg-Ganea groups constructed by M. Fluch and I. Leary, are $\calFn$-Eilenberg-Ganea groups for all $n\geq 0$.  Let $G$ be a Fluch-Leary group. From \cite{BLN01}, $G$ is a $\calF_0$-Eilenberg-Ganea group. On the other hand, since $G$ is hyperbolic, any abelian free subgroup of $G$ has rank at most $1$, hence $\calFn=\vcyc$ for all $n\geq1$. We can use \Cref{corollary:finite:edge:stab} to construct $\calFn$-Eilenberg-Ganea groups with $\vcyc\neq \calFn$ for $n\geq 2$. For instance, we can take as vertex groups of $\mathbf{Y}$ some Fluch-Leary groups and at least one free abelian group of rank $\leq n$  to get a $\calFn$-Eilenberg-Ganea group such that $\vcyc$ is a proper subfamily of $\calFn$, provided $n\geq2$.

\end{remark}

\subsection{The family of (virtually) polycyclic subgroups}

All the results of this section hold if you replace the word polycyclic by virtually polycyclic.

Recall a group $H$ is said to be \emph{polycyclic} if there exist a finite chain  of subgroups 
\[1=H_0 \trianglelefteq H_1 \trianglelefteq\cdots \trianglelefteq H_{n-1} \trianglelefteq H_n=H,\]
such that each factor group $H_i/H_{i-1}$ is cyclic. The number $r$, of infinite cyclic factor groups, is a well-defined invariant of $H$ and is called the \emph{Hirsch rank} of $H$.

Consider the following family

\[
\calP_n = \{ K\leq G | K \text{ is polycyclic of Hirsch rank at most }n  \}
\]

\begin{theorem}\label{thm:polycyclic}
Let $\mathbf{Y}$ be a graph of groups  with fundamental group $G$. Let $n\geq 0$ be an integer.  Assume the following conditions hold:
\begin{enumerate}[label=(\roman*)]
    \item $\mathbf{Y}$ is an $\calP_n$-admissible.
    \item The splitting of $G$ given by $\mathbf{Y}$ is $\calP_{n-1}$-acylindrical.
\end{enumerate}
Then $G$ is a $\calP_n$-Eilenberg-Ganea group.
\end{theorem}
\begin{proof}
We will see that this theorem is a consequence of \Cref{main:thm}. By \Cref{remark:slender}(3), every virtually polycyclic group acting on a tree either fixes a vertex or it acts cocompactly on a geodesic. Thus, in notation of \Cref{main:thm} we have $\calFn=\calFn_0\sqcup \calFn_1$.

Moreover every element of $\calFn_1$ is a virtually cyclic subgroup by hypothesis (2) in \Cref{main:thm}.

In the statement of \Cref{main:thm} we choose $\calP=\calP_{n-1}$. Every $\calP_{n-1}$-by-$(\dbZ$ or $D_\infty)$ belongs to $\calP_n$ because a polycyclic-by-polycyclic group is again polycyclic, and the Hirsch rank is additive with respect to extensions. Finally, any $\calP_{n-1}$-by-$\dbZ/2$ lives in $\calP_{n-1}$. Thus we have verified all hypothesis of \Cref{main:thm}, and the result follows. 

\end{proof}

\begin{remark}
Let $H$ be one of the Fluch-Leary examples. Then the family $V\calP_n$ of virtually polycyclic subgroups of $H$ coincides with the family $\vcyc$ of virtually cyclic subgroups. In fact, from \cite[Example~5.26]{Lu05} we have that for every virtually polycyclic group $P$ of Hirsch rank $r$, $\vcd(P)=\underline{\gd}(P)=r$. By \cite[Theorem~2]{BLN01} we have $\vcd(P)\leq \cdfin(P)\leq \gdfin(P)$. Hence $\cdfin(P)=\underline{\gd}(P)=r$.   By monoticity of the $\fin$-cohomological dimension $\cdfin(P)\leq \cdfin(H)=2$ , we get that every virtually polycyclic subgroup of $H$ has rank at most $2$. On the other hand, every virtually polycyclic subgroup of $H$ of rank $2$ contains a $\dbZ$-by-$\dbZ$ subgroup, hence it contains a $\dbZ^2$ subgroup, which is impossible because $H$ is hyperbolic. Thus every virtually polycyclic subgroup of $H$ has rank at most 1, i.e. it is virtually cyclic.

Thus we can construct $V\calP_n$-Eilenberg-Ganea groups for every $n\geq1$. Indeed consider $\mathbf{Y}$ a graph of groups with the following characteristics: we place $H$ in one of the vertex, finite group in the edges and we fill every other vertex with groups in such a way that $\mathbf{Y}$ is $V\calP_n$-admissible.
\end{remark}

\bibliographystyle{alpha} 
\bibliography{myblib}
\end{document}